\documentclass[a4paper,11pt]{amsart}

\usepackage[english]{my-shortcuts}
\usepackage{srcltx,algorithm,algorithmic,bbm,dsfont}
\usepackage{graphicx}   

\title[Estimation of ARMA models via convex optimization]
{Joint estimation and model order selection for one dimensional ARMA models via convex optimization: 
a nuclear norm penalization approach}
\author{St\'ephane Chr\'etien, Tianwen Wei {\tiny and} Basad Ali Hussain Al-sarray} \thanks{Laboratoire de Math\'ematiques, UMR 6623,
Universit\'e de Franche-Comt\'e, 16 route de Gray,
25030 Besancon, France. Email: stephane.chretien@univ-fcomte.fr}

\begin{document}
\maketitle


%
%
%

\begin{abstract}
The problem of estimating ARMA models is computationally interesting due to the nonconcavity of the log-likelihood function. Recent results were based on the
convex minimization. Joint model selection using penalization by a convex norm, e.g. the nuclear norm of a certain 
matrix related to the state space formulation was extensively studied from a computational viewpoint. The goal of the present short note 
is to present a theoretical study of a nuclear norm penalization based variant of the method of \cite{Bauer:Automatica05,Bauer:EconTh05} under the assumption of 
a Gaussian noise process.   
\end{abstract}

{\bf Keywords:} ARMA models, Time series, Low rank model, Prediction, Nuclear norm penalization.

\bigskip

\section{Introduction}

The Auto-regressive with moving average (ARMA) model is central to the field of time serie analysis and has been studied since the early thirties in the field of econometrics
\cite{Shumway:Springer14}. ARMA time series are sequences of the form $(x_{t})_{t\in \mathbb{N}}$ satisfying the following recursion 
\begin{eqnarray}
x_{t} & = & \sum^{p}_{i=1} a_{i}x_{t-i}+\sum^{q}_{j=1}b_{j}e_{t-j}+e_{t} \label{eq:Eq.1} 
\end{eqnarray}
for all $t\geq max \left\lbrace p,q\right\rbrace $, and we focus on the case where $(e_{t})_{t\in \mathbb{N}}$ is a sequence of zero mean 
independent identically distributed Gaussian random variables with variance denoted by $\sigma_\epsilon^2$ for simplicity
\footnote{Extentions to more sophisticated models for the noise $(\epsilon_t)$ in order to accomodate more applications were also studied extensively in the recent years but the Gaussian case is already a challenge from the algorithmic perspective as will be discussed below}. As is well known \cite{Shumway:Springer14},
time series model are adequate for a wide range of phenomena in economics, engineering, social science, epidemiology, ecology, signal processing, etc.
They can also be helpful as a building block in more complicated models such as GARCH models, which are particularly useful in financial time series analysis.

Two problems are to be addressed when studying ARMA time series:
\begin{enumerate}
\item estimate $p$ and $q$, the intrinsic orders of the model. 
\item estimate $a=(a_{1},a_{2},..., a_{p})$ and $b=(b_{1},b_{2},...b_{q})$.
\end{enumerate}

In the case where $q=0$, the convention is to write \eqref{eq:Eq.1} as:
\begin{eqnarray}
x_{t}&=&\sum^{p}_{i=1}a_{i}x_{t-1}+e_{t}
\end{eqnarray}
and $x_{t}$ to simply called an AR process. Estimation of $a$ is often performed using the conditional likelihood approach, given $x_{0},...,x_{p-1}$ yielding to the standard Yule-Walker equations. On the other hand, the model order selection problem is often performed using a penalized log-likelihood approach such as AIC,BIC,.., may also use the plain likelihood. We refer the reader to the standard text of Brockwell and Davis for more details on these standard problems.
Turning back to the full ARMA model, it is well known that the log-likelihood is not a concave function, and that multiple stationary points exist which can lead to severe bias when using local optimization routines for such as gradient or Newton-type methods for the joint estimation of $a$ and $b$. In Shumway and Stoffer 
\cite{Shumway:Springer14} and iterative procedure resembling the EM algorithm is proposed, which seems more appropriate for the ARMA model than standard optimization algorithms. However, no convergence guarantee towards a global maximizer is provided. Concerning the model selection problem, 
penalties play a prominent role in modern statistical theory and practice, in particular since the recent successes of the LASSO in regression and its multiple generalization. The nuclear norm penalization has played an import for many problems in engineering, machine learning and statistics such as matrix completion, \ldots
Application of nuclear norm penalization to state space model estimation and model order selection using a moment-like estimator in a convex optimization framework 
is proposed in \cite{FAZ}. The approach of \cite{FAZ} is a remarkable contribution since convex model selection and 
state space estimation were combined for the first time in the problem of Time Series. However the approach of \cite{FAZ} is supported by no theoretical guarantee
yet. Another approach for State Space model estimation was proposed in \cite{Bauer:Automatica05,Bauer:EconTh05} where good practical performances are reported and an 
asymptotic analysis is provided. This method as well as the unpenalized version of the method in \cite{FAZ} 
can be recast into the family of subspace methods; see \cite{Verh}. 
In such subspace-type methods, model order selection and model estimation are decoupled and it is natural to wonder if the 
approach of \cite{Bauer:Automatica05} can be refined in order to incorporate joint model selection using  a nuclear norm penalty as in \cite{FAZ}. 

Based on the evidence of the practical efficiency of subspace-type methods \cite{Verh},  
our goal in the present note is to propose a theoretical study of a nuclear norm penalized version of the subspace method from \cite{Bauer:Automatica05} 
which incorporates the main ideas in \cite{FAZ}.

\section{The subspace method}
\subsection{Recall on the subspace approach}

A real valued random discrete dynamical system $(x_t)_{t\in \mathbb N}$ admits a State Space representation if there
exists a discrete time process $(s_t)_{t\in \mathbb N}$ such that  
\begin{eqnarray*}
    s_{t+1} & = & As_{t}+Ke_{t}\\
    x_{t}   & = & Bs_{t}+e_{t}\label{eq:Eq.2} 
\end{eqnarray*}
where $(e_{t})_{t\in \mathbb N}$ is the noise, and $A\in \mathbb{R}^{p\times p}$, $B\in \mathbb{R}^{1\times p}$,
$K\in \mathbb{R}^{p\times 1}$ are parameter matrices. It is well known that ARMA processes 
admit a State Space representation and vice versa \cite{Shumway:Springer14}. 

\subsection{Prediction}
The problem of predicting $x_{t+j}$ for $j\geq 0$ based on the knowledge of $x_{t^\prime}$, $t^\prime <t$ and $s_{t}$ 
can be solved easily following the approach by Bauer \cite{Bauer:Automatica05,Bauer:EconTh05}. 
For given initial values $x_{0}$, $e_{0}$, the State Space representation gives
\begin{eqnarray*}
x_{t+h} & = & e_{t+h}+\sum^{h}_{j=1} BA^{j-1}Ke_{t+h-j}+BA^{h}s_{t}
\end{eqnarray*}
On the other hand, the State Space representation implies that 
\begin{eqnarray*}
s_{t} & = & As_{t-1}+Ke_{t-1}\\
& = & As_{t-1}+K\left( x_{t-1}-Bs_{t-1}\right)   \\
& = & \left( A-KB\right)s_{t-1}+Kx_{t-1} \\
& = & \cdots \\
\end{eqnarray*}
Thus, we obtain 
\begin{eqnarray*}
s_t & = & \left( A-KB\right)^{t}s_{0}+\sum_{j=0}^{t-1}\left(A-KB\right)^{j} Kx_{t-1-j}. 
\end{eqnarray*}

In what follows, we will assume that we observe $x_0,\ldots,x_T$ and that $t>0$ is such that $T-2t+1>0$.

\subsection{Prediction with Hankel matrices}
We will rewrite the prediction problem in terms of some Hankel matrices. For this purpose, define 
\begin{eqnarray*}
\bar{A}= A-KB, \quad \bar{\mathcal A}_0=[\bar{A}^ts_0,\bar{A}^{t+1}_0,\ldots,\bar{A}^{T-t+1}s_0], \quad 
\mathcal K & = & \left[\bar{A}^{t-1}K,\cdots,\bar{A}^{2}K, \bar{A}K, K\right],
\end{eqnarray*}  
\begin{eqnarray*}
\mathcal O & =  
\left[
\begin{array}{c}
B \\
BA \\
\vdots \\
BA^{t-1}
\end{array}
\right]
\quad \textrm{ and }\quad 
\mathcal N = & 
\left[
\begin{array}{cccccc}
1 & 0 & \cdots & \cdots & \cdots & 0 \\
BK & 1 & 0 & \cdots & \cdots & 0 \\
\vdots & \vdots & \vdots & \vdots & \vdots & \vdots \\
BA^{t-2}K & BA^{t-3}K & \cdots & \cdots & BK & 1 
\end{array}
\right].
\end{eqnarray*}
Then, we have 
\begin{eqnarray}
\left[
\begin{array}{c}
x_t \\
\vdots \\
x_{2t-1}
\end{array}
\right]
& = & 
\mathcal O s_t + 
\mathcal N \left[
\begin{array}{c}
e_{t} \\
\vdots \\
e_{t+h}
\end{array}
\right]
\label{1}
\end{eqnarray}
and 
\begin{eqnarray}
s_t & = & \mathcal K 
\left[
\begin{array}{c}
x_0 \\
\vdots \\
x_{t-1}
\end{array}
\right]+ 
\left( A-KB\right)^{t}s_{0}.
\label{2}
\end{eqnarray}
Combining (\ref{1}) and (\ref{2}), we thus obtain 
\begin{eqnarray*}
\left[
\begin{array}{c}
x_t \\
\vdots \\
x_{2t-1}
\end{array}
\right]
& = & 
\mathcal O \mathcal K 
\left[
\begin{array}{c}
x_0 \\
\vdots \\
x_{t-1}
\end{array}
\right]+ 
\mathcal O \left( A-KB\right)^{t}s_{0} + 
\mathcal N \left[
\begin{array}{c}
e_{t} \\
\vdots \\
e_{t+h}
\end{array}
\right].
\end{eqnarray*}
Now, define
\begin{eqnarray*}
X_{past} & =  \left[ \begin{array}{cccc}
 x_{0}   &  x_{1}  &  \cdots  & x_{T-2t+1}  \\
 x_{1}   &  x_{2}  &  \cdots  & x_{T-2t+2}    \\
 \vdots  &  \vdots &  \vdots  & \vdots      \\
 x_{t-1}   & x_{t} &  \cdots  & x_{T-t}     \\ 
\end{array}\right]
\quad \textrm{ and } \quad 
X_{future}  = & \left[ \begin{array}{cccc}
x_{t}  &  x_{t+1}   &\cdots  &  x_{T-t+1}  \\
x_{t+1}  &  x_{t+2}   &\cdots  &  x_{T-t+2}   \\
\vdots   &  \vdots    &\vdots  &  \vdots       \\
x_{2t-1}   &  x_{2t}  &\cdots  &  x_{T}     \\ 
\end{array}\right]. 
\end{eqnarray*}
Both matrices are Hankel matrices. The first one represents the past values and and second one the future values. 
Define also the noise matrix 
\begin{eqnarray*}
E & = & \left[ \begin{array}{cccc}
e_{t}  &  e_{t+1}   &\cdots  &  e_{T-t+1}  \\
e_{t+1}  &  e_{t+2}   &\cdots  &  e_{T-t+2}   \\
\vdots   &  \vdots    &\vdots  &  \vdots       \\
e_{2t-1}   &  e_{2t}  &\cdots  &  e_{T}     \\ 
\end{array}\right]. 
\end{eqnarray*}
All these Hankel matrices are related by the following equation
\begin{eqnarray*}
X_{future} & = & \mathcal O \mathcal K \: X_{past} + \mathcal O \bar{\mathcal A}_0+ \mathcal N E.
\end{eqnarray*}

\section{The estimation problem}
In the last section, we showed that the matrices $A$, $B$ and $K$ of the State Space model entered nicely into
an equation allowing prediction of future values based on past values of the dynamical system. Our goal is now 
to use this equation  to estimate the matrices $A$, $B$ and $C$. One interesting feature of our procedure 
is that the dimension $p$ of the State Space model can be estimated jointly with the matrices themselves. 

\subsection{Estimating $\mathcal O \mathcal K$}
The matrix $\mathcal O \mathcal K$ can be estimated using a least squares approach corresponding to solving 
\begin{eqnarray}
\min_{L\in \mathbb R^{t\times t}} \: \frac12 \Vert X_{future}-L\: X_{past} \Vert_F^2.  
\label{LS}
\end{eqnarray}
This procedure will make sense if the term $\mathcal O \bar{\mathcal A}_0$ is small. 
This can indeed be justified if $t$ is large and if $\Vert\bar{A}\Vert$ is small. Let us call $\hat{L}$ 
a solution of (\ref{LS}).

\subsection{Nuclear Norm penalized $\ell_1$-norm for low rank estimation}
An interesting property of the matrix $\mathcal O\mathcal K$ is that its rank is the State's dimension $p$ when
$A$ has full rank. Moreover, $\mathcal O\mathcal K$ has small rank compared to $t$ when $t$ is large compared to $p$. 
Therefore, one is tempted to penalize the least squares problem (\ref{LS}) with a low-rank promoting penalty. 

One option is to try to solve 
\begin{eqnarray}
\label{rankLS}
\min_{L\in \mathbb R^{t\times t}} \: \frac{1}{2}\Vert X_{future}-L\: X_{past}\Vert^{2}_{F}+\lambda \: {\rm rank}
\left( L\right) 
\end{eqnarray}
The main drawback of this approach is that the rank function is non continuous and non convex. This 
renders the optimization problem intractable in practice. 
Fortunately, the rank function admits a well known convex surrogate, which is the nuclear norm, i.e. the sum of the singular values, denoted by $\Vert .\Vert_{*}$.\\
Thus, a nice convex relaxation of (\ref{rankLS}) is given by
\begin{eqnarray}
\label{nucLS}
\min_{L\in \mathbb R^{t\times t}}\: \frac{1}{2}\Vert X_{future}-L\: X_{past}\Vert^{2}_{F}+\lambda \: \Vert L\Vert_{*}. 
\end{eqnarray} 
It has been observed in practice that nuclear norm penalized least squares provide low rank solution 
for many interesting estimation problems \cite{RechtEtAl:SIREV0?}.  

\section{Main results}
\label{theory}
The penalized least-squares problem (\ref{nucLS}) can be transformed into the following constrained problem 
\begin{eqnarray}
\label{constrainedform}
\min_{L\in \mathbb R^{t\times t}} \: \Vert L\Vert_* \quad \textrm{ subject to } \quad \Vert X_{future}-L\: X_{past}\Vert_{F}\le \eta,                                                                                                                                                                                                                                                                                                                                                                                                                                                                                                                                                                                                                                                                                                                                                                                                                                                                                                                    
\end{eqnarray}
for some appropriate choice of $\eta$. 

Let $\Sigma$ denote the covariance matrix of $[x_0,\ldots,x_{t-1}]^t$ and let $\Sigma^{\pm \frac12}$ denote the square root of 
$\Sigma^{\pm 1}$. Then,
Let $H$ be the random matrix whose components are given by 
\begin{eqnarray*}
H_{s,r} & = & \sum_{s^\prime=0}^{T-2t+1} \epsilon_{s,s^\prime} z_{s^\prime+r}.
\end{eqnarray*}
where $\epsilon_{s,s^\prime}$, $s=0,\ldots,t-1$ and $s^\prime=0,\ldots,T-2t+1$ are independent Rademacher random variables which 
are independent of $z_{s^\prime}$, $s^\prime=0,\ldots,T-t$.
Let $\Sigma^H$ denote the covariance matrix of ${\rm vec}(H)$. 
Let $\mathcal M$ denote the operator defined by 
\begin{eqnarray}
\mathcal M & = & {\rm Mat}\left(\Sigma^{H^{-1/2}}{\rm vec}(\cdot)\right)
\label{Mdef}
\end{eqnarray}
and let $\mathcal M^{-*}$ denote the adjoint of the inverse of $\mathcal M$. The fact that 
$\mathcal M$ is invertible is easily obtained (see Section \ref{conseq}) 
and is seen from the fact that $\Sigma^{H}$ has all its eigenvalues 
equal to $T-2t+1$ according to Section \ref{specSigH}. 
Let $\mathcal S$ be the operator defined by
\begin{eqnarray*}
\mathcal S(\cdot) &\mapsto & \mathcal M^{-*} \left(\cdot \right) \: \Sigma^{-1/2}.
\end{eqnarray*}
and let $\mathcal T$ be the mapping 
\begin{eqnarray*}
\mathcal T(\cdot) \mapsto \frac1{\sqrt{t\: (T-2t+2)}} \: \mathcal M^{-1} \: \left(\cdot\:\:\Sigma^{\frac12}\right).
\end{eqnarray*}

Our main result is the following theorem.

\begin{theo}
\label{main}
Let $\xi$ be any positive real number. Assume that $\eta$ is such that 
\begin{eqnarray}
\Vert\mathcal O \bar{\mathcal A} \: s_0+ \mathcal N E \Vert & \ge & \eta
\end{eqnarray}
with probability less than or equal to $e^{-\nu^2/2}$ for some $\nu>0$. Then, with probability greater than or equal to $1-e^{-\nu^2/2}$, 
\begin{eqnarray}
\Vert \mathcal O \mathcal K -\hat{L}\Vert_F & \le & \frac{2\eta}
{\Lambda},
\end{eqnarray}
where
\begin{eqnarray*}
& & \Lambda \ge \xi \sqrt{t(T-2t+1)} \left(1-
\frac{4 \xi} {\sqrt{\pi}} \left(\frac{e}{2} \right)^{\frac{1}4}  
\sigma_{\max}\left(\Sigma^{1/2}\right)\: \sqrt{t}\right) \\
\\
& & \hspace{1cm}- 2\sqrt{2}\: \sqrt{\frac{t}{T-2t+1}}\: \sqrt{\frac{\sigma_{\max}(\Sigma)}{\sigma_{\min}(\Sigma) }}
\:  
\sqrt{\left( \left(2\: ct+1\right)
+c \: \sqrt{t}\right) \: \sqrt{t} \: \frac{{\rm rank}(\mathcal O\mathcal K)}{c\: \sqrt{\sigma_{\min}(\Sigma)}} 
+ 2 \: t}- \nu \xi. 
\end{eqnarray*}
\end{theo}
In the remainder of this section, we introduce the results, notations and tools for proving this theorem. The proof 
is given in Section \ref{pf}. 

\subsection{Some notations}
For all $s=0,\ldots,t-1$ and $s^\prime=0,\ldots,T-2t+1$, 
let $\mathcal A_{s,s^\prime}$ denote the operator defined by 
\begin{eqnarray}
\mathcal A_{s,s^\prime}(L) & = & \sum_{r=0}^{t-1} L_{s,r} x_{s^\prime+r} 
\end{eqnarray}
and let $\mathcal A$ denote the operator 
\begin{eqnarray}
L \mapsto \left(\mathcal A_{s,s^\prime}(L)\right)_{s=1,\ldots,t,\: s^\prime=0,\ldots,T-2t+1}.
\end{eqnarray}
The descent cone of the nuclear norm at $\mathcal O\mathcal K$, denoted by $\mathcal D(\Vert\cdot\Vert_*,\mathcal O\mathcal K)$, 
is defined by 
\begin{eqnarray}
\mathcal D(\Vert\cdot\Vert_*,\mathcal O\mathcal K) & = & 
\cup_{\tau>0}\: \left\{D \in \mathbb R^{t\times t} \mid \Vert H+\tau D\Vert_* \le \Vert H\Vert_*  \right\}. 
\end{eqnarray} 

\subsection{A deterministic inequality}

The following result will be the key of our analysis.  
\begin{theo} {\bf \cite{Tropp:Renaissance15}}
\label{deter}
Assume that 
\begin{eqnarray}
\Vert\mathcal O \bar{\mathcal A} \: s_0+ \mathcal N E \Vert & \le & \eta.
\end{eqnarray}
Let $\hat{L}$ denote any solution of (\ref{constrainedform}). Then, 
\begin{eqnarray}
\Vert \mathcal O \mathcal K -\hat{L}\Vert_F & \le & \frac{2\eta}{\lambda_{\min}\left(\mathcal A,
\mathcal D(\Vert\cdot\Vert_*,\mathcal O\mathcal K) \right)},
\end{eqnarray}
where 
\begin{eqnarray}
\lambda_{\min}\left(\mathcal A,\mathcal D(\Vert\cdot\Vert_*,\mathcal O\mathcal K)\right) 
& = & \min_{\stackrel{\Vert D\Vert_F=1,}{D\in \mathcal D(\Vert\cdot\Vert_*,\mathcal O\mathcal K)}} 
\quad \Vert \mathcal A(D)\Vert_F. 
\end{eqnarray}
\end{theo}

\subsection{A lower bound on $\lambda_{\min}\left(\mathcal A,\mathcal D(\Vert\cdot\Vert_*,\mathcal O\mathcal K)\right)$}
We will closely follow the approach of Tropp based on Mendelson's bound. For this purpose, we will need the 
definition of the Gaussian mean width $w_G(\mathfrak X)$ of a set $\mathfrak X\in \mathbb R^d$
\begin{eqnarray*}
w_G(\mathfrak X) & = & \mathbb E\left[\sup_{x\in \mathfrak X} \: \langle G,x \rangle\right],
\end{eqnarray*}
where the expectation is taken with respect to the Gaussian random vector $G$ taking values in $\mathbb R^d$. 
The statistical dimension of $\mathfrak X$ (see e.g. \cite{ALMT14}) is 
Let us also denote by $Q_{\xi}$ the quantity 
\begin{eqnarray*}
Q_{\xi}(D) & = & \frac1{t}\sum_{s=0}^{t-1} \: \bP \left(\left\vert\sum_{r=0}^{t-1} D_{s,r} z_{s^\prime+r}\right\vert\ge \xi\right),
\end{eqnarray*}
which, as one might easily check, does not depend on $s^\prime$. Recall that $\Sigma$ is the covariance matrix of 
$[x_0,\ldots,x_{t-1}]^t$ and that $\Sigma^{\pm \frac12}$ denotes the square root of $\Sigma^{\pm 1}$. Thus, 
\begin{eqnarray*}
\left[
\begin{array}{c}
z_0\\
\vdots \\
z_{t-1}
\end{array}
\right]
:=\Sigma^{-\frac12}
\left[
\begin{array}{c}
x_0\\ 
\vdots \\
x_{t-1}
\end{array}
\right]
\end{eqnarray*}
follows the standard Gaussian distribution $\mathcal N(0,I)$. Let $\tilde{D}=D\Sigma^{\frac12}$. We now state Tropp's result. 
\begin{lemm}
\label{tropp}
Define
\begin{eqnarray*}
K & = & \frac1{\sqrt{t\: (T-2t+2)}} \: \mathcal M^{-1}\left(\Sigma^{-1/2} \: \mathcal D(\Vert\cdot\Vert_*,\mathcal O\mathcal K)\right).
\end{eqnarray*}
We have 
\begin{eqnarray*}
\lambda_{\min}\left(\mathcal A,\mathcal D(\Vert\cdot\Vert_*,\mathcal O\mathcal K)\right) 
& \ge & 
\quad \xi \sqrt{t(T-2t-2)} \quad \inf_{\Vert \tilde{D}\Sigma^{1/2} \Vert_F=1} 
\quad Q_{2\xi}(\tilde{D}) - \frac{2 }{\sigma_{\min}(\mathcal S)}\: w_G(K)- \nu \xi 
\end{eqnarray*}
with probability greater than or equal to $1-\exp(-\nu^2/2)$.
\end{lemm}
\begin{proof}
See Section \ref{prftropp}. 
\end{proof}

\subsection{A lower bound on $\inf_{\Vert \tilde{D}\Sigma^{1/2} \Vert_F=1} \: Q_{2\xi}(\tilde{D})$ }
Since 
\begin{align*}
Z & =\sum_{r=0}^{t-1} \tilde{D}_{s,r} z_{s^\prime+r}
\end{align*}
follows the law $\mathcal N(0,\sum_{r=0}^{t-1} \tilde{D}_{s,r}^2)$, using Lemma \ref{cki} from the Appendix, we get 
\begin{align*}
\bP\left( Z^2 \le \left(\sum_{r=0}^{t-1} \tilde{D}_{s,r}^2\right) \sqrt{u}\right)  
& \le \frac2{\sqrt{\pi}} \left(e\ \frac{u}{2} \right)^{\frac{1}4}.
\end{align*}
Thus, setting 
\begin{align*}
u & =\frac{\xi^4}{\left(\sum_{r=0}^{t-1} \tilde{D}_{s,r}^2\right)^2}, 
\end{align*}
we obtain   
\begin{align*}
\bP \left(\left\vert\sum_{r=0}^{t-1} \tilde{D}_{s,r} z_{s^\prime+r}\right\vert\ge \xi\right) & \ge  
1-\frac2{\sqrt{\pi}} \left(\frac{e}{2} \right)^{\frac{1}4} \frac{\xi}{\sqrt{\sum_{r=0}^{t-1} \tilde{D}_{s,r}^2}}.
\end{align*}
This finally gives 
\begin{align*}
Q_{2\xi}(\tilde{D}) & \ge 1-
\frac{4 \xi} {\sqrt{\pi}} \left(\frac{e}{2} \right)^{\frac{1}4} \frac1{t}\sum_{s=0}^{t-1} \frac{1}{\sqrt{\sum_{r=0}^{t-1} \tilde{D}_{s,r}^2}}.
\end{align*}
Let us now compute a lower bound to the infimum of this quantity over the set of 
$\tilde{D}$ satisfying $\Vert \tilde{D}\Sigma^{1/2} \Vert_F=1$. For this purpose, first note that 
\begin{align*}
\inf_{\Vert \tilde{D}\Sigma^{1/2} \Vert_F=1} \: Q_{2\xi}(\tilde{D}) & \ge 
1-\sup_{\Vert \tilde{D} \Vert_F\ge \sigma_{\max}\left(\Sigma^{1/2}\right)^{-1}} \: 
\frac{4 \xi} {\sqrt{\pi}} \left(\frac{e}{2} \right)^{\frac{1}4} \frac1{t}\sum_{s=0}^{t-1} 
\frac{1}{\sqrt{\sum_{r=0}^{t-1} \tilde{D}_{s,r}^2}}.
\end{align*}
On the other hand, simple manipulations of the optimality conditions using symmetry prove that  
\begin{align*}
\sup_{\Vert A \Vert_F \le 1} \: \frac1{t}\sum_{s=0}^{t-1} 
\frac{1}{\sqrt{\sum_{r=0}^{t-1} A_{s,r}^2}} & = \sqrt{t}.
\end{align*}
Therefore, 
\begin{align}
\inf_{\Vert \tilde{D}\Sigma^{1/2} \Vert_F=1} \: Q_{2\xi}(\tilde{D}) & \ge 1-
\frac{4 \xi} {\sqrt{\pi}} \left(\frac{e}{2} \right)^{\frac{1}4}  
\sigma_{\max}\left(\Sigma^{1/2}\right)\: \sqrt{t}. 
\label{Q}
\end{align}

\subsection{The Gaussian mean width of $K$}
The Gaussian mean width of a set $\mathfrak X$ and its statistical dimension are related by 
\begin{align}
w_G(\mathfrak X)^2 & \le \delta(\mathfrak X) \le w_G(\mathfrak X)^2+1. 
\label{compwidthdim}
\end{align}
See \cite[Proposition 10.2]{ALMT14} for a proof. 
In this subsection, we estimate the Gaussian mean width of $K$ using its statistical dimension. 

\subsubsection{The descent cone $\mathcal D(\Vert\cdot\Vert_*,\mathcal O\mathcal K)$}
The descent cone of the nuclear norm satisfies \cite[Eq. (4.1)]{Tropp:Renaissance15} which we recall now 
\begin{eqnarray}
\label{Troppo}
\mathcal D(\Vert\cdot\Vert_*,\mathcal O\mathcal K)^\circ & = & \overline{{\rm cone}}
\left(\partial \Vert \cdot \Vert_*(\mathcal O\mathcal K) \right). 
\end{eqnarray}
\subsubsection{Computation of $K^\circ$}
Using Proposition 4.2 in \cite{Tropp:Renaissance15}, we obtain 
\begin{eqnarray}
\sup_{\stackrel{\Vert \tilde{D}^*\Vert_F=1,}{\tilde{D}^*\in K}} \quad  \: 
\left\langle \tilde{D}^*,\tilde{H}\right\rangle & \le &
{\rm dist}\left(\tilde{H},K^\circ \right). 
\label{Troppo2}
\end{eqnarray}
We now have to compute the polar cone of $K$. We have 
\begin{eqnarray*}
K^\circ & = & \left\{\Delta \mid \langle \Delta,D \rangle\le 0 \quad \forall \: D\: \in K\right\} \\
& = & \left\{\Delta \mid \langle \frac1{\sqrt{t\: (T-2t+2)}} \: \mathcal M^{-} \: 
\left(\Delta\:\Sigma^{\frac12}\right),D \rangle\le 0 \quad \forall \: D\: \in \mathcal D(\Vert\cdot\Vert_*,\mathcal O\mathcal K)\right\}. 
\end{eqnarray*}
Recall that $\mathcal T$ is the mapping 
\begin{eqnarray*}
\Delta \mapsto \frac1{\sqrt{t\: (T-2t+2)}} \: \mathcal M^{-1} \: \left(\Delta\:\Sigma^{\frac12}\right).
\end{eqnarray*}
Then, we obtain that 
\begin{eqnarray*}
K^\circ & = & \mathcal T^{-1}\left(\mathcal D(\Vert\cdot\Vert_*,\mathcal O\mathcal K)^\circ \right).  
\end{eqnarray*}

\subsubsection{An upper bound on the statistical dimension of $K$}
Let us write the singular value decomposition of $\mathcal O \mathcal K$ as 
\begin{eqnarray*}
\mathcal O \mathcal K & = & 
\left[ 
\begin{array}{cc}
U_1 & U_2 
\end{array}
\right]
\left[
\begin{array}{cc}
{\rm diag}(\sigma_{\mathcal O\mathcal K}) & 0\\
0 & 0 
\end{array}
\right]
\left[ 
\begin{array}{cc}
V_1 & V_2 
\end{array}
\right]^t
\end{eqnarray*}
where $\sigma_{\mathcal O\mathcal K}$ is the vector of the singular values of $\mathcal O \mathcal K$. Moreover, the subdifferential of the Schatten norm is 
given by 
\begin{eqnarray*}
\partial \Vert \cdot \Vert_*(\mathcal O\mathcal K) & = & 
\left[ 
\begin{array}{cc}
U_1 & U_2 
\end{array}
\right]
\left\{
\left[ 
\begin{array}{cc}
I & 0 \\
0 & Y 
\end{array}
\right]\mid \Vert Y\Vert \le 1 \right\}
\left[ 
\begin{array}{cc}
V_1 & V_2 
\end{array}
\right]^t.
\end{eqnarray*}
Therefore, using (\ref{Troppo2}), we obtain that 
\begin{eqnarray*}
\mathbb E\left[\left(\sup_{\stackrel{\Vert \tilde{D}^*\Vert_F=1,}{\tilde{D}^*\in K}} \quad  \: 
\left\langle \tilde{D}^*,\tilde{H}\right\rangle\right)^2\right] & \le & 
\mathbb E\left[\min_{\tau>0, \: \Vert Y\Vert\le 1}\Vert \mathcal T^{-1}\left(
\tau 
\left[ 
\begin{array}{cc}
U_1V_1^t & 0 \\
0 & U_2YV_2^t 
\end{array}
\right]  
\right)
- 
\tilde{H}\Vert_F^2\right].
\end{eqnarray*}
Thus, we get 
\begin{align*}
\mathbb E\left[\left(\sup_{\stackrel{\Vert \tilde{D}^*\Vert_F=1,}{\tilde{D}^*\in K}} \quad  \: 
\left\langle \tilde{D}^*,\tilde{H}\right\rangle\right)^2\right] & \le  
\mathbb E\Bigg[\min_{\tau>0, \: \Vert Y\Vert\le 1} \Vert \mathcal T^{-1}\Vert \: \Bigg(
\tau^2 \Vert U_1V_1^t \Vert_F^2 +\Vert \tau U_2YV_2^t - \mathcal T_{2,2}(\tilde{H})\Vert_F^2 \\
& \hspace{2cm}  + \Vert \mathcal T_{1,2}(\tilde{H}) \Vert_F^2+ \Vert \mathcal T_{2,1}(\tilde{H}) \Vert_F^2\Bigg)
\Bigg]
\end{align*}
where 
\begin{align*}
\mathcal T & = 
\left[
\begin{array}{cc}
\mathcal T_{11} & \mathcal T_{12} \\
\mathcal T_{21} & \mathcal T_{22}
\end{array}
\right],
\end{align*}
and the dimension of $T_{11}$ is ${\rm rank}(\mathcal O\mathcal K)\times {\rm rank}(\mathcal O\mathcal K)$ and 
the dimension of $\mathcal T_{j,j^\prime}$ for all other combinations of $j$ and $j^\prime$ is easily deduced from the dimension of $\mathcal T$.  
which gives, after taking $\tau=\Vert U_2^t\mathcal T_{2,2}(\tilde{H}) V_2\Vert$, 
\begin{align*}
\mathbb E\left[\left(\sup_{\stackrel{\Vert \tilde{D}^*\Vert_F=1,}{\tilde{D}^*\in K}} \quad  \: 
\left\langle \tilde{D}^*,\tilde{H}\right\rangle\right)^2\right] & \le  
 \sigma_{\min}(\mathcal T)^{-1} \: \Bigg(\mathbb E\left[\tau^2\right] \: {\rm rank}(\mathcal O\mathcal K) \\
& \hspace{2cm} + \Bigg(\sigma_{\max}(\mathcal T_{1,2})^2 + \sigma_{\max}(\mathcal T_{2,1})^2\Bigg) \: \mathbb E\Bigg[\Vert \tilde{H}\Vert_F^2\Bigg]\Bigg).
\end{align*}
Note that
\begin{align*}
\tau & \le \Vert \mathcal T_{2,2}\Vert \Vert \tilde{H}\Vert.
\end{align*}
By Gordon's theorem \cite[Theorem 10.2]{Vershynin:ArXiv14}, $\mathbb E\left[\Vert \tilde{H}\Vert\right] \le 2\: \sqrt{t}$.
Moreover, by Lemma \ref{bordel} in the Appendix,  
\begin{align*}
\mathbb E\left[\Vert \tilde{H}\Vert^2 \right] 
& \le \frac2{c} \left(2\: ct+1\right)+2 \: \sqrt{t}.
\end{align*}
On the other hand, $\mathbb E\left[ \Vert \tilde{H} \Vert_F^2 \right]=2t$. Therefore, we obtain that 
\begin{align*}
\delta (K) & = \mathbb E\left[\left(\sup_{\stackrel{\Vert \tilde{D}^*\Vert_F=1,}{\tilde{D}^*\in K}} \quad  \: 
\left\langle \tilde{D}^*,\tilde{H}\right\rangle\right)^2\right]\\
 & \le  2\: \sigma_{\min}(\mathcal T)^{-1} \: \Bigg(\frac1{c} \: \Vert \mathcal T_{2,2}\Vert^2 \left( \left(2\: ct+1\right)
+c \: \sqrt{t}\right) \: {\rm rang}(\mathcal O\mathcal K) \\
& \hspace{1.3cm} + 2\: \Bigg(\sigma_{\max}(\mathcal T_{1,2})^2 +  \sigma_{\max}(\mathcal T_{2,1})^2\Bigg) \: t\Bigg).
\end{align*}
Using (\ref{compwidthdim}), we obtain that 
\begin{eqnarray}
& w_G(K) \le & \label{wGK}\\
\nonumber \\
& \sqrt{2\: \sigma_{\min}(\mathcal T)^{-1} \: \Bigg(\frac1{c} \: \Vert \mathcal T_{2,2}\Vert^2 \left( \left(2\: ct+1\right)
+c \: \sqrt{t}\right) \: {\rm rang}(\mathcal O\mathcal K) 
+ 2\: \Bigg(\sigma_{\max}(\mathcal T_{1,2})^2 +  \sigma_{\max}(\mathcal T_{2,1})^2\Bigg) \: t\Bigg)} &
\nonumber 
\end{eqnarray}

\subsection{Proof of Theorem \ref{main}}
\label{pf}
Combining Lemma \ref{tropp} with (\ref{Q}) and (\ref{wGK}), we obtain that 
\begin{eqnarray*}
\lambda_{\min}\left(\mathcal A,\mathcal D(\Vert\cdot\Vert_*,\mathcal O\mathcal K)\right) 
& \ge & 
\quad t\: \sqrt{T-2t-2} \quad \frac{4 \xi^2} {\sqrt{\pi}} \left(\frac{e}{2} \right)^{\frac{1}4}  
\sigma_{\min}\left(\Sigma^{1/2}\right) \\
& & \hspace{-4cm} - \frac{2\sqrt{2} }{\sigma_{\min}(\mathcal S)}\:  \sqrt{\Bigg(\frac1{c} \: \Vert \mathcal T_{2,2}\Vert^2 \left( \left(2\: ct+1\right)
+c \: \sqrt{t}\right) \: \frac{{\rm rang}(\mathcal O\mathcal K)}{\sigma_{\min}(\mathcal T)} 
+ \Bigg(\sigma_{\max}(\mathcal T_{1,2})^2 +  \sigma_{\max}(\mathcal T_{2,1})^2\Bigg) \: t\Bigg)}- \nu \xi 
\end{eqnarray*} 
Using that 
\begin{align*}
\Vert \mathcal T_{2,2}\Vert^2 & \le \Vert \mathcal T\Vert^2,
\end{align*}
and
\begin{align*}
\sigma_{\max}(\mathcal T_{1,2})^2 +  \sigma_{\max}(\mathcal T_{2,1})^2 & \le 2 \: \Vert \mathcal T\Vert^2,
\end{align*}
and combining this last inequality with Theorem \ref{deter}, we obtain the following proposition.
\begin{prop}
\label{main}
Let $\xi$ be any positive real number. Assume that $\eta$ is such that 
\begin{eqnarray}
\Vert\mathcal O \bar{\mathcal A} \: s_0+ \mathcal N E \Vert & \ge & \eta
\end{eqnarray}
with probability less than or equal to $e^{-\nu^2/2}$ for some $\nu>0$. Then, with probability greater than or equal to $1-e^{-\nu^2/2}$, 
\begin{eqnarray}
\Vert \mathcal O \mathcal K -\hat{L}\Vert_F & \le & \frac{2\eta}
{\Lambda},
\end{eqnarray}
where
\begin{eqnarray*}
& & \Lambda \ge \xi \sqrt{t(T-2t+1)} \left(1-
\frac{4 \xi} {\sqrt{\pi}} \left(\frac{e}{2} \right)^{\frac{1}4}  
\sigma_{\max}\left(\Sigma^{1/2}\right)\: \sqrt{t}\right) \\
\\
& & \hspace{1cm}- \frac{2\sqrt{2} \Vert \mathcal T\Vert }{\sigma_{\min}(\mathcal S)}\:  \sqrt{\left( \left(2\: ct+1\right)
+c \: \sqrt{t}\right) \: \frac{{\rm rang}(\mathcal O\mathcal K)}{c\: \sigma_{\min}(\mathcal T)} 
+ 2 \: t}- \nu \xi. 
\end{eqnarray*}
\end{prop}
Combinig this result with the bounds from Section \ref{conseq}, the proof is completed.   

\section{Conclusion}

The goal of the present note is to show that the performance of nuclear norm penalized subspace-type methods 
can be studied theoretically. We concentrated on a special approach due to Bauer \cite{Bauer:Automatica05}. 
Our approach can easily be extended to the case of the method promoted in \cite{FAZ}. Our next objective 
for future research is to address the case of more general noise sequences such as in \cite{FrancqZakoian:JSPI98}.   

\section{Appendix: Technical intermediate results}
In this section, we gather some technical results used in the proof of Theorem \ref{main}. 

\subsection{Proof of Lemma \ref{tropp}}
\label{prftropp}
\subsubsection{First step}
We have  
\begin{eqnarray}
\lambda_{\min}\left(\mathcal A,\mathcal D(\Vert\cdot\Vert_*,\mathcal O\mathcal K)\right) 
& = & \min_{\stackrel{\Vert D\Vert_F=1,}{D\in \mathcal D(\Vert\cdot\Vert_*,\mathcal O\mathcal K)}}\: 
\Vert\mathcal A(D) \Vert_F \\
& = & \min_{\stackrel{\Vert D\Vert_F=1,}{D\in \mathcal D(\Vert\cdot\Vert_*,\mathcal O\mathcal K)}}\: 
\left(\sum_{s=0}^{t-1} \sum_{s^\prime=0}^{T-2t+1} \:\: \left(\sum_{r=0}^{t-1} D_{s,r} x_{s^\prime+r}\right)^2\right)^{\frac12} 
\end{eqnarray}
Recall that $\Sigma$ is the covariance matrix of $[x_0,\ldots,x_{t-1}]^t$ and that $\Sigma^{\pm \frac12}$ denotes the square root of 
$\Sigma^{\pm 1}$. Thus, 
\begin{eqnarray*}
\left[
\begin{array}{c}
z_0\\
\vdots \\
z_{t-1}
\end{array}
\right]
:=\Sigma^{-\frac12}
\left[
\begin{array}{c}
x_0\\ 
\vdots \\
x_{t-1}
\end{array}
\right]
\end{eqnarray*}
follows the standard Gaussian distribution $\mathcal N(0,I)$. Recall also that $\tilde{D}=D\Sigma^{\frac12}$. Then, we have 
\begin{eqnarray}
\lambda_{\min}\left(\mathcal A,\mathcal D(\Vert\cdot\Vert_*,\mathcal O\mathcal K)\right) 
& = & \min_{\stackrel{\Vert \tilde{D}\Sigma^{-1/2}\Vert_F=1,}{\tilde{D}\in \mathcal D(\Vert\cdot\Vert_*,\mathcal O\mathcal K)
\Sigma^{1/2}}}\: 
\left(\sum_{s=0}^{t-1} \sum_{s^\prime=0}^{T-2t+1} \:\: \left(\sum_{r=0}^{t-1} \tilde{D}_{s,r} z_{s^\prime+r}\right)^2\right)^{\frac12} 
\end{eqnarray}
Now, we have 
\begin{eqnarray*}
\left(\frac1{t\: (T-2t+1)}\quad 
\sum_{s=0}^{t-1} \sum_{s^\prime=0}^{T-2t+1} \:\: \left(\sum_{r=0}^{t-1} \tilde{D}_{s,r} z_{s^\prime+r}\right)^2\right)^{\frac12}
& \ge & \frac1{t\: (T-2t+1)}\quad 
\sum_{s=0}^{t-1} \sum_{s^\prime=0}^{T-2t+1} \:\: \left\vert\sum_{r=0}^{t-1} \tilde{D}_{s,r} z_{s^\prime+r}\right\vert
\end{eqnarray*}
which gives, by Markov's inequality 
\begin{eqnarray*}
\left(\frac1{t\: (T-2t+1)}\quad 
\sum_{s=0}^{t-1} \sum_{s^\prime=0}^{T-2t+1} \:\: \left(\sum_{r=0}^{t-1} \tilde{D}_{s,r} z_{s^\prime+r}\right)^2\right)^{\frac12}
& \ge & \frac{\xi}{t\: (T-2t+1)}\quad 
\sum_{s=0}^{t-1} \sum_{s^\prime=0}^{T-2t+1} \:\: \mathds{1}
\left\{\left\vert\sum_{r=0}^{t-1} \tilde{D}_{s,r} z_{s^\prime+r}\right\vert\ge \xi\right\}. 
\end{eqnarray*}
Thus, we obtain 
\begin{eqnarray*}
& & \left(\frac1{t\: (T-2t+1)}\quad 
 \sum_{s=0}^{t-1} \sum_{s^\prime=0}^{T-2t+1} \:\: \left(\sum_{r=0}^{t-1} \tilde{D}_{s,r} z_{s^\prime+r}\right)^2\right)^{\frac12} \\
& & \hspace{1cm}\ge \quad \xi Q_{2\xi}(\tilde{D}) - \frac{\xi}{t\: (T-2t+1)}\quad 
\sum_{s=0}^{t-1} \sum_{s^\prime=0}^{T-2t+1} \:\: \left( Q_{2\xi}(\tilde{D}) -\mathds{1}
\left\{\left\vert\sum_{r=0}^{t-1} \tilde{D}_{s,r} z_{s^\prime+r}\right\vert\ge \xi\right\}\right). 
\end{eqnarray*}

\subsubsection{Second step}
Let
\begin{eqnarray*}
f(z_0,\ldots,z_{T-t}) & = & 
\sup_{\stackrel{\Vert \tilde{D}\Sigma^{-1/2}\Vert_F=1,}{\tilde{D}\in \mathcal D(\Vert\cdot\Vert_*,\mathcal O\mathcal K)\Sigma^{1/2}}} 
\quad 
\sum_{s=0}^{t-1} \sum_{s^\prime=0}^{T-2t+1} \:\: \left( Q_{2\xi}(\tilde{D}) -\mathds{1}
\left\{\left\vert\sum_{r=0}^{t-1} \tilde{D}_{s,r} z_{s^\prime+r}\right\vert\ge \xi\right\}\right).
\end{eqnarray*}
We will now use the bounded difference inequality to control this quantity. For this purpose, 
notice that 
\begin{eqnarray*}
\vert f(\zeta_0,\ldots,\zeta_s,\ldots,\zeta_{T-t})-f(\zeta_0,\ldots,\zeta^\prime_s,\ldots,\zeta_{T-t})\vert 
& \le & 2 \: t\: (T-2t+2). 
\end{eqnarray*} 
for all $(\zeta_0,\ldots,\zeta_s,\ldots,\zeta_{T-t})$ in $\mathbb R^{T-t+1}$ and $\zeta^\prime_s \in \mathbb R$. Thus, 
\begin{eqnarray*}
f(z_0,\ldots,z_{T-t}) - \mathbb E \left[f(z_0,\ldots,z_{T-t})\right] & \le & \nu \: \sqrt{t\: (T-2t+2)},
\end{eqnarray*}
with probability $1-e^{-\nu^2/2}$ for all $\nu\in \mathbb R_+$. Now, the expected supremum can be 
bounded in the same manner as in \cite[Equation 5.6]{Tropp:Renaissance15}. 
\begin{eqnarray*}
\mathbb E \left[f(z_0,\ldots,z_{T-t})\right] & \le & 
\frac2{\xi} \: \mathbb E \left[ \sup_{\stackrel{\Vert \tilde{D}\Sigma^{-1/2}\Vert_F=1,}{\tilde{D}\in \mathcal D(\Vert\cdot\Vert_*,\mathcal O\mathcal K)\Sigma^{1/2}}}
\quad \sum_{s=0}^{t-1} \sum_{s^\prime=0}^{T-2t+1} \epsilon_{s,s^\prime} \sum_{r=0}^{t-1} \tilde{D}_{s,r} z_{s^\prime+r}
\right]
\end{eqnarray*}
where $\epsilon_{s,r}$, $s=0,\ldots,t-1$ and $r=0,\ldots,t-1$ are independent Rademacher random variables which 
are independent of $z_{s^\prime}$, $s^\prime=0,\ldots,T-t$. Therefore, we obtain 
\begin{eqnarray*}
& & \inf_{\stackrel{\Vert \tilde{D}\Sigma^{-1/2}\Vert_F=1,}{\tilde{D}\in \mathcal D(\Vert\cdot\Vert_*,\mathcal O\mathcal K)\Sigma^{1/2}}}
\left(\frac1{t\: (T-2t+1)}\quad 
 \sum_{s=0}^{t-1} \sum_{s^\prime=0}^{T-2t+1} \:\: \left(\sum_{r=0}^{t-1} \tilde{D}_{s,r} z_{s^\prime+r}\right)^2\right)^{\frac12} \\
& & \hspace{1cm}\ge \quad \xi Q_{2\xi}(\tilde{D}) - \frac{\xi}{t\: (T-2t+1)}\quad 
\Bigg(\frac2{\xi} \: \mathbb E \left[ \sup_{\tilde{D} \in \Sigma^{-1/2} \: \mathcal D(\Vert\cdot\Vert_*,\mathcal O\mathcal K)}
\quad \sum_{s=0}^{t-1} \sum_{s^\prime=0}^{T-2t+1} \epsilon_{s,s^\prime} \sum_{r=0}^{t-1} \tilde{D}_{s,r} z_{s^\prime+r}
\right]\\
& & \hspace{3cm} + \nu \: \sqrt{t\: (T-2t+2)}\Bigg), 
\end{eqnarray*}
which gives 
\begin{eqnarray*}
& & \inf_{\stackrel{\Vert \tilde{D}\Sigma^{-1/2}\Vert_F=1,}{\tilde{D}\in \mathcal D(\Vert\cdot\Vert_*,\mathcal O\mathcal K)\Sigma^{1/2}}}
\left(\sum_{s=0}^{t-1} \sum_{s^\prime=0}^{T-2t+1} \:\: \left(\sum_{r=0}^{t-1} \tilde{D}_{s,r} z_{s^\prime+r}\right)^2\right)^{\frac12} \\
& & \hspace{2cm}\ge \quad \xi \: \sqrt{t\: (T-2t+2)} \: Q_{2\xi}(\tilde{D}) \\  
& & \hspace{3cm}  
-2 \: \mathbb E \left[ \sup_{\stackrel{\Vert \tilde{D}\Sigma^{-1/2}\Vert_F=1,}{\tilde{D}\in \mathcal D(\Vert\cdot\Vert_*,\mathcal O\mathcal K)\Sigma^{1/2}}}
\quad \frac1{\sqrt{t\: (T-2t+2)}} 
\sum_{s=0}^{t-1} \sum_{s^\prime=0}^{T-2t+1} \epsilon_{s,s^\prime} \sum_{r=0}^{t-1} \tilde{D}_{s,r} z_{s^\prime+r}
\right] - \nu \xi. 
\end{eqnarray*}
Let us denote by $W$ the quantity 
\begin{eqnarray*}
W & = & \mathbb E \left[ \sup_{\stackrel{\Vert \tilde{D}\Sigma^{-1/2}\Vert_F=1,}{\tilde{D}\in \mathcal D(\Vert\cdot\Vert_*,\mathcal O\mathcal K)\Sigma^{1/2}}}
\quad \frac1{\sqrt{t\: (T-2t+2)}} 
\sum_{s=0}^{t-1} \sum_{s^\prime=0}^{T-2t+1} \epsilon_{s,s^\prime} \sum_{r=0}^{t-1} \tilde{D}_{s,r} z_{s^\prime+r}
\right]. 
\end{eqnarray*}
Then, we have 
\begin{eqnarray*}
W & = & \mathbb E \left[ \sup_{\stackrel{\Vert \tilde{D}\Sigma^{-1/2}\Vert_F=1,}{\tilde{D}\in \mathcal D(\Vert\cdot\Vert_*,\mathcal O\mathcal K)\Sigma^{1/2}}}
\quad \frac1{\sqrt{t\: (T-2t+2)}} \: \langle \tilde{D},H\rangle \right],
\end{eqnarray*}
where we recall that $H$ is the random matrix whose components are given by 
\begin{eqnarray*}
H_{s,r} & = & \sum_{s^\prime=0}^{T-2t+1} \epsilon_{s,s^\prime} z_{s^\prime+r}
\end{eqnarray*}
and $\Sigma^H$ denotes the covariance matrix of ${\rm vec}(H)$. Let $\tilde{H}=
\mathcal M(H)$ where $\mathcal M$ denotes the operator defined by 
\begin{eqnarray*}
\mathcal M(\cdot) & = & {\rm Mat}\left(\Sigma^{H^{-1/2}}{\rm vec}(\cdot)\right).
\end{eqnarray*}
Then $\tilde{H}$ is a Gaussian matrix with i.i.d. components with law $\mathcal N(0,1)$. 
Using the invertibility of $\mathcal M$ proved in Section \ref{conseq}, we get 
\begin{eqnarray*}
W & = \mathbb E \left[ \sup_{\stackrel{\Vert \mathcal M^{-*} \left(\tilde{D}^*\: \Sigma^{-1/2}\right)
\Vert_F=1,}{\tilde{D}^*\in K}}
\quad  \: \left\langle \tilde{D}^*,\tilde{H}\right\rangle \right],
\end{eqnarray*}
where 
\begin{eqnarray*}
K & = & \frac1{\sqrt{t\: (T-2t+2)}} \: \mathcal M^{-*} \: 
\left(\mathcal D(\Vert\cdot\Vert_*,\mathcal O\mathcal K)\:\Sigma^{\frac12}\right),
\end{eqnarray*}
where we recall that $\mathcal M^{-*}$ is the adjoint of the inverse of $\mathcal M$.  
Moreover, we have 
\begin{eqnarray*}
\sup_{\stackrel{\Vert \mathcal M^{-*} \left(\tilde{D}^*\right)
\: \Sigma^{-1/2}\Vert_F=1,}{\tilde{D}^*\in K}}
\quad  \: 
\left\langle \tilde{D}^*,\tilde{H}\right\rangle 
& \le & \frac1{\sigma_{\min}(\mathcal S)}\quad \sup_{\stackrel{\Vert \tilde{D}^*\Vert_F=1,}{\tilde{D}^*\in K}}
\quad  \: \left\langle \tilde{D}^*,\tilde{H}\right\rangle 
\end{eqnarray*}
where $\sigma_{\min}(\mathcal S)$ is the smallest singular value of the operator $\mathcal S$ defined by
\begin{eqnarray*}
\mathcal S(\cdot) & \mapsto & \mathcal M^{-*} \left(\cdot \right) \: \Sigma^{-1/2}.
\end{eqnarray*}
Thus, 
\begin{align*}
W & \le \frac{w_G(K)}{\sigma_{\min}(\mathcal S)} 
\end{align*}
and the proof is completed. 

\subsection{Control of $\mathbb E\left[\Vert \tilde{H}\Vert^2 \right]$}

\begin{lemm}
\label{bordel}
We have 
\begin{align*}
\mathbb E\left[\Vert \tilde{H}\Vert^2 \right] 
& \le \left(1+\frac{1}{2 \: ct}\right)\mathbb E\left[\Vert \tilde{H}\Vert\right]^2+\mathbb E[\Vert \tilde{H}\Vert ].
\end{align*}
\end{lemm}
\begin{proof}
By Gaussian concentration \cite[Proposition 4]{Tao:Whatsnew09} and the fact that 
the spectral (operator) norm is 1-Lipschitz, we obtain that for all $u>0$, 
\begin{align*}
\mathbb P \left( \Vert \tilde{H}\Vert\ge \mathbb E \left[\Vert \tilde{H} \Vert \right] +u \right) & \le e^{-c u^2}
\end{align*}
for some absolute positive constant $c$. Taking $u=\delta \mathbb E \left[\Vert \tilde{H} \Vert \right]$, we obtain that 
\begin{align*}
\mathbb P \left( \Vert \tilde{H}\Vert\ge (1+\delta) \: \mathbb E \left[\Vert \tilde{H} \Vert \right] \right) & \le e^{-4 \: c \delta^2 t}.
\end{align*} 
Thus, 
\begin{align*}
\mathbb E\left[\Vert \tilde{H}\Vert^2 \right] & = \int_0^{+\infty} \mathbb P\left( \Vert \tilde{H}\Vert^2 \ge s \right) d s \\
& = \int_0^{\mathbb E\left[\Vert \tilde{H}\Vert\right]^2} \mathbb P\left( \Vert \tilde{H}\Vert^2 \ge s \right) d s 
+ \int_{\mathbb E\left[\Vert \tilde{H}\Vert\right]^2}^{+\infty} \mathbb P\left( \Vert \tilde{H}\Vert^2 \ge s \right) d s \\
& = \mathbb E\left[\Vert \tilde{H}\Vert\right]^2 
+ \int_{\mathbb E\left[\Vert \tilde{H}\Vert\right]^2}^{+\infty} \mathbb P\left( \Vert \tilde{H}\Vert \ge \sqrt{s} \right) d s \\
& \le \mathbb E\left[\Vert \tilde{H}\Vert\right]^2 
+ \int_{\mathbb E\left[\Vert \tilde{H}\Vert\right]^2}^{+\infty} \exp \left(-4 \: c 
\left(\frac{\sqrt{s}-\mathbb E[\Vert \tilde{H}\Vert ]}{\mathbb E[\Vert \tilde{H}\Vert ]}\right)^{2}\: t\right)\: d s
\end{align*}
and making the change of variable $r=(\sqrt{s}-\mathbb E[\Vert \tilde{H}\Vert ])^2$, we obtain 
\begin{align*}
\mathbb E\left[\Vert \tilde{H}\Vert^2 \right] & = \int_0^{+\infty} \mathbb P\left( \Vert \tilde{H}\Vert^2 \ge s \right) d s \\
& \le \mathbb E\left[\Vert \tilde{H}\Vert\right]^2 
+ \int_{0}^{+\infty} \exp \left(-4 \:  
\frac{ct}{\mathbb E[\Vert \tilde{H}\Vert ]^2} \: r\right)\: \left(1+\frac1{\sqrt{r}}\right)\: d r \\
& \le \mathbb E\left[\Vert \tilde{H}\Vert\right]^2 
+ 2 \: \int_{0}^{+\infty} \exp \left(-4 \:  
\frac{ct}{\mathbb E[\Vert \tilde{H}\Vert ]^2} \: r\right)\: d r + 
\int_{0}^{\mathbb E[\Vert \tilde{H}\Vert ]^2} \frac1{\sqrt{r}}\: d r.
\end{align*}
Thus, we obtain 
\begin{align*}
\mathbb E\left[\Vert \tilde{H}\Vert^2 \right] 
& \le \mathbb E\left[\Vert \tilde{H}\Vert\right]^2+\mathbb E[\Vert \tilde{H}\Vert ] 
-   
\frac{\mathbb E[\Vert \tilde{H}\Vert ]^2}{2 \: ct}\: \left[ \exp \left(-4 \:  
\frac{ct}{\mathbb E[\Vert \tilde{H}\Vert ]^2} \: r\right)\right]_0^{+\infty} \\
& \le \left(1+\frac{1}{2 \: ct}\right)\mathbb E\left[\Vert \tilde{H}\Vert\right]^2+\mathbb E[\Vert \tilde{H}\Vert ].
\end{align*}
This completes the proof.
\end{proof}

\subsection{Some properties of $\Sigma$, $\Sigma^H$, $\mathcal M$, $\mathcal S$ and $\mathcal T$}
\subsubsection{The spectrum of $\Sigma$}

The spectrum of $\Sigma$ can be studied using the methods of Grenander and Szego \cite{GrenanderSzego58}. In 
\cite{PalmaBondon:SPL03}, the classical results are extended to the case of generalized fractional processes. 
it was shown in particular by Grenander and Szego in \cite[Chapter 5]{GrenanderSzego58} that $2 \pi m \le \lambda \le 2 \pi M$
for any eigenvalue $\lambda$ of $\Sigma$, where $m$ and $M$ are the essential infimum and supremum of the spectral density function $f$ of the process. For ARMA 
processes, this function is just 
\begin{align*}
f(\nu) & =  \frac{\sigma_\epsilon^2}{2\pi} \left\vert \frac{\theta(e^{i\nu})}{\phi(e^{i\nu})}\right\vert^2
\end{align*}
where 
\begin{align*}
\phi(z) & = 1-a_1z-\cdots-a_pz^p \textrm{ and } \theta(z) = 1+b_1z+\cdots+b_qz^q.
\end{align*}
   
\subsubsection{The spectrum of $\Sigma^H$}
\label{specSigH}
Recall that $H$ is the random matrix whose components are given by 
\begin{eqnarray*}
H_{s,r} & = & \sum_{s^\prime=0}^{T-2t+1} \epsilon_{s,s^\prime} z_{s^\prime+r}. 
\end{eqnarray*}
where $\epsilon_{s,s^\prime}$, $s=0,\ldots,t-1$ and $s^\prime=0,\ldots,T-2t+1$ are independent Rademacher random variables which 
are independent of $z_{s^\prime}$, $s^\prime=0,\ldots,T-t$.

Using matrix representation, we have
\bean
H &=& \epsilon z  \\
&=& \begin{pmatrix}
\epsilon_{0,1} & \epsilon_{0,2} & \cdots & \epsilon_{0, T-2t+1}  \\
\epsilon_{1,1} & \epsilon_{1,2} & \cdots & \epsilon_{1, T-2t+1}  \\
\vdots & \vdots & \ddots  & \vdots  \\
\epsilon_{t-1,1} & \epsilon_{t-1,2} & \cdots & \epsilon_{t-1,T-2t+1}  
\end{pmatrix}
\begin{pmatrix}
z_{0} & z_{1} & \cdots & z_{t-1}  \\
z_{1} & z_{2} & \cdots & z_{t}  \\  
\vdots& \vdots& \ddots & \vdots  \\ 
z_{T-2t+1}& z_{T-2t+2}& \cdots & z_{T-t}
\end{pmatrix}. 
\eean
Let $Z_p$ be the $(p+1)$-th column of $z$. Then
\bean
{\rm vec} (H)&=& 
\begin{pmatrix}
\epsilon Z_0 \\
\epsilon Z_{1} \\
 \vdots \\
  \epsilon Z_{t-1}
\end{pmatrix}.
\eean
The $(p,q)$-th block of $\Sigma^H$ is given by
\bean
\Sigma^H_{[p,q]} = \mathbb E[\epsilon E[Z_p Z_q^t] \epsilon^t],
\eean
where for $p<q$
\bean
\mathbb E[Z_p Z_q^t] = 
\begin{pmatrix}
0 & 0 \\
I_{T-2t+1 - (q-p)} & 0 
\end{pmatrix}.
\eean
Here, $I_{T-2t+1 - (q-p)}$ denotes the identity matrix of dimension $T-2t+1 - (q-p)$.

Partitioning $\epsilon$  appropriately as
\bean
\epsilon = \begin{pmatrix} \epsilon_{[1,1]} & \epsilon_{[1,2]} \\ \epsilon_{[2,1]} & \epsilon_{[2,2]}   \end{pmatrix},
\eean
we deduce that
 \bean
\Sigma^H_{[p,q]} &=& 
E\left[
\begin{pmatrix} \epsilon_{[1,1]} & \epsilon_{[1,2]} \\ \epsilon_{[2,1]} & \epsilon_{[2,2]}   \end{pmatrix}
\begin{pmatrix}
0 & 0 \\
I_{T-2t+1 - (q-p)} & 0 
\end{pmatrix}
\begin{pmatrix} \epsilon_{[1,1]}^t & \epsilon_{[2,1]}^t \\ \epsilon_{[1,2]}^t & \epsilon_{[2,2]}^t   \end{pmatrix}
\right] \\
&=&
E \begin{pmatrix} 
\epsilon_{[1,2]}\epsilon_{[1,1]}^t & \epsilon_{[1,2]}\epsilon_{[2,1]}^t \\
\epsilon_{[2,2]}\epsilon_{[1,1]}^t & \epsilon_{[2,2]}\epsilon_{[2,1]}^t   
\end{pmatrix} \\
&=&0
 \eean
for $p <q$. Similarly, we can show that $\Sigma^H_{[p,q]}=0$ for $p>q$. As for $p=q$, we have 
$E[Z_p Z_p^t] = I_{T-2t+1}$. Thus
\bean
\Sigma^H_{[p,p]} = E[\epsilon \epsilon^t] = (T-2t+1)I_{t}
\eean
It is then follows that $\Sigma^H = (T-2t+1) I_{t(T-2t+1)}$.

\subsubsection{Consequences for $\mathcal M$, $\mathcal S$ and $\mathcal T$}
\label{conseq}

Recall that $\mathcal M$ denotes the operator defined by 
\begin{eqnarray}
\mathcal M & = & {\rm Mat}\left(\Sigma^{H^{-1/2}}{\rm vec}(\cdot)\right)
\label{Mdef}
\end{eqnarray}
and $\mathcal M^{-*}$ denotes the adjoint of the inverse of $\mathcal M$. Using the 
resuts of Section \ref{specSigH}, we obtain that 
\begin{eqnarray}
\mathcal M & = & \frac1{\sqrt{T-2t+1}}\ Id. 
\label{Mdefbis}
\end{eqnarray}
and 
\begin{eqnarray}
\mathcal M^{-*} & = & \sqrt{T-2t+1}\ Id. 
\label{Mdefbis}
\end{eqnarray}

Using these results, we obtain that $\mathcal S$ is the operator defined by
\begin{eqnarray*}
\mathcal S(\cdot) & \mapsto & \frac1{\sqrt{T-2t+1}}\: \cdot  \: \Sigma^{-1/2}.
\end{eqnarray*}
and $\mathcal T$ is the mapping 
\begin{eqnarray*}
\mathcal T(\cdot) \mapsto \frac1{\sqrt{t}} \: \cdot\:\:\Sigma^{\frac12}.
\end{eqnarray*}
We thus have the following results on $\mathcal T$.
\begin{eqnarray*}
\| \mathcal T\| & \le & \frac{\sqrt{\sigma_{\max}(\Sigma)}}{\sqrt{t}}
\end{eqnarray*}
and 
\begin{eqnarray*}
\sigma_{\min} (\mathcal T) & \ge & \frac{\sqrt{\sigma_{\min}(\Sigma)}}{\sqrt{t}}.
\end{eqnarray*}
We also obtain that  
\begin{eqnarray*}
\sigma_{\min}(\mathcal S) & \ge & \frac{\sqrt{\sigma_{\min}(\Sigma)}}{\sqrt{T-2t+1}} .
\end{eqnarray*}

\subsection{Some properties of the $\chi^2$ distribution}

We recall the following useful
bounds for the $\chi^2(\nu)$ distribution of degree of freedom $\nu$.
\begin{lemm} {\rm \cite[Lemma B.1]{ChretienDarses:IEEEIT14}} The following bounds hold:
\label{cki}
\bean
\bP\left(\chi(\nu) \ge  \sqrt{\nu}+ \sqrt{2t} \right) & \le & \exp(-t)\\
\bP\left(\chi(\nu) \le \sqrt{u \nu}\right)  & \le & \frac2{\sqrt{\pi \nu}}
\left(u\ e/2\right)^{\frac{\nu}4}.
\eean
\end{lemm}


\begin{thebibliography}{1}

\bibitem{ALMT14} Amelunxen, D. Lotz, M., McCoy, M.B. and Tropp, J., 
Living on the edge: Phase transitions in conve programs with random data, http://arxiv.org/abs/1303.6672.

\bibitem{Bauer:Automatica05} Bauer, D., Asymptotic properties of subspace estimators. Automatica, 41, no. 3, (2005), 359-376.

\bibitem{Bauer:EconTh05} Bauer, D. (2005). Estimating linear dynamical systems using subspace methods. Econometric Theory, 21(01), 181-211.

\bibitem{ChretienDarses:IEEEIT14} Chretien, S., and Darses, S. (2014). Sparse recovery with unknown variance: a LASSO-type approach. Information Theory, IEEE Transactions on, 60(7), 3970-3988.

\bibitem{PalmaBondon:SPL03} Palma, W., and Bondon, P., On the eigenstructure of generalized fractional processes. 
Statistics \& Probability Letters 65 (2003) 93--101.

\bibitem{FAZ} Fazel, M., Pong, T. K. and Sun, D. (2009), 
Hankel matrix rank minimization with applications to system identification and realization.

\bibitem{Gordon} Gordon, Y., Some inequalities for Gaussian processes and applications, 
Israel Journal of Mathematics. 11/1985; 50(4):265-289. 

\bibitem{GrenanderSzego58} Grenander, U., Szego, G., Toeplitz forms and their applications. California Monographs in Mathematical Sciences.
University of California Press, Berkeley, CA (1958).

\bibitem{FrancqZakoian:JSPI98} Francq, C., and Zako\"ian, J. M. (1998). Estimating linear representations of nonlinear processes. Journal of Statistical Planning and Inference, 68(1), 145-165.

\bibitem{Over} Van Overschee, P. and De Moor, B. (1994), N4SID*: Subspace algorithms for the identification of combined deterministic stochastic system, Automatica, \textbf{27}, 75--93.

\bibitem{RechtEtAl:SIREV0?} Recht, B., Fazel, M. and Parrilo, P., Guaranteed minimum rank solutions of matrix equations via nuclear
norm minimization. SIAM Review, 200?.

\bibitem{Shumway:Springer14} Shumway, R.H., and Stoffer, D.S., Time series analysis and its applications, EZ - Third edition 
{\em http://www.stat.ualberta.ca/~wiens/stat479/tsa3EZ.pdf}

\bibitem{Tao:Whatsnew09} Tao, T., https://terrytao.wordpress.com/2009/06/09/talagrands-concentration-inequality.

\bibitem{Tropp:Renaissance15} Tropp, J., Convex recovery of a structured signal from independent random linear measuremenents, 
http://arxiv.org/abs/1405.1102.

\bibitem{Verh} Verhaegen, M. and Verdult, V. (2007), Filtering and system identification a least squares approach, 
Cambridge University Press.

\bibitem{Vershynin:ArXiv14} Vershynin, Estimation in high dimensions: a geometric perspective, http://arxiv.org/abs/1405.5103.


\end{thebibliography}
\end{document}